\documentclass{amsproc}
\usepackage{amssymb}

\theoremstyle{plain}
 \newtheorem{theorem}{Theorem}[section]
 \newtheorem{proposition}{Proposition}[section]
 \newtheorem{lemma}{Lemma}[section]
 
\theoremstyle{definition}
 
 \newtheorem{definition}{Definition}[section]
\theoremstyle{remark}
 
 \numberwithin{equation}{section}

\renewcommand{\leq}{\leqslant}
\renewcommand{\geq}{\geqslant}

\title[Scheduling Asynchronous Round-Robin Tournaments]{Scheduling Asynchronous Round-Robin Tournaments}

\keywords{timetabling, sports scheduling, round-robin tournament, fairness}

\author[Warut Suksompong]{Warut Suksompong}

\address{
Department of Computer Science \\ 
Stanford University   \\ 
Stanford, CA 94305\\
USA}
\email{warut@cs.stanford.edu}

\begin{document}

\vspace{18mm} \setcounter{page}{1} \thispagestyle{empty}

\begin{abstract}
We study the problem of scheduling asynchronous round-robin tournaments. We consider three measures of a schedule that concern the quality and fairness of a tournament. We show that the schedule generated by the well-known ``circle design'' performs well with respect to all three measures when the number of teams is even, but not when the number of teams is odd. We propose a different schedule that performs optimally with respect to all measures when the number of teams is odd.
\end{abstract}

\maketitle

\section{Introduction}

A round-robin tournament, also known as an all-play-all tournament, is a popular format for organizing sports competitions. In a round-robin tournament, every pair of teams play each other a fixed number of times during the competition. Since every team competes with every other team, the winner of a round-robin tournament is usually thought to depend much less on luck than that of, say, a knockout tournament. A series of work has investigated how to schedule a round-robin tournament when different notions are central to the organizers' consideration. One line of research has focused on time-relaxed tournaments, which takes into account the issue of time off between games involving the same team \cite{Knust08,Knust10,SchonbergerMaKo04}, while another has considered fairness issues \cite{Briskorn09,BriskornKn10,VanthofPoBr10,ZengMi13}. We refer the interested reader to a survey by Rasmussen and Trick \cite{RasmussenTr08} and a book by Anderson \cite{Anderson97} for more detail on the literature.

In this paper, we study the problem of scheduling asynchronous round-robin tournaments, i.e., round-robin tournaments in which no two games take place at the same time. There are a number of reasons why it might be desirable to schedule all games at different times. Indeed, this tournament format allows spectators to follow all the games live, and the organizers can maximize revenue while having to organize the same number of games. Tournaments may even need to be asynchronous if there is only one venue where a game can take place. An example of an asynchronous round-robin tournament is the 2012 Premier League Snooker in England, in which five players in the group stage play a total of ten games in ten different weeks (albeit in ten different venues as well).

When scheduling an asynchronous round-robin tournament, the organizers may desire properties that improve the quality and fairness of the tournament. Unlike in knockout tournaments, where the organizers can significantly impact the outcome of the tournament by setting up a bracket of their choice (see, e.g., \cite{KimSuVa16,Vassilevska10}), the set of games to be played in a round-robin tournament cannot be changed. Nevertheless, the order in which the games are played can still be an important factor in a round-robin tournament. For example, when teams have a longer rest between games, they are more likely to have a relaxing rest and perform at their full potential in the next game. On the other hand, if some team has a long rest going into a game while its opponent has just played its previous game, the former team could be at a clear advantage. Another desirable property of a schedule is that at any point during the tournament, all teams should have played roughly the same number of games. This prevents the advantage of knowing too many results involving other teams and the possibility of collusion as well. We define measures that capture all of these properties, and we exhibit schedules that perform (close to) optimally with regard to our measures. In particular, we show that the schedule generated by the well-known ``circle design'' performs well with respect to all three measures when the number of teams is even, but not so well when the number of teams is odd. We also propose a different schedule that performs optimally with respect to all three measures when the number of teams is odd. We hope that this schedule will be of practical interest to organizers of asynchronous round-robin tournaments.

A related problem that is worth mentioning is the problem of finding balanced tournament designs, which has been considered by some prior work \cite{BlestFi88,HaselgroveLe77,SchellenbergVa77}. In the setting of balanced tournament designs, it is assumed that there exist external factors that make some games different from others, and it is desirable that teams receive roughly the same effect from these external factors. For instance, the tournament might involve games during different times of the day or at different venues. Since some teams might be more familiar with playing in the morning than in the evening or with playing at one venue than another, the aim of a balanced tournament design is to eliminate or minimize the potential advantage by scheduling teams to play as evenly across the different times and venues as possible. On the other hand, in our setting there is no inherent difference between games. Indeed, a good example to keep in mind throughout this paper is that the games in the tournament are scheduled on consecutive days, one game per day, at a single venue.

A summary of our results can be found in Table \ref{table:summary}.

\begin{table}
    \begin{tabular}{| p{35mm} | p{23mm} | p{21mm} | p{23mm} | p{21mm} |}
    \hline
     & Circle method, $n$ even & Any schedule, $n$ even & Circle method, $n$ odd & Any schedule, $n$ odd \\ \hline
     Guaranteed rest time & $(n-4)/2$ & $\leq (n-4)/2$ & $(n-5)/2$ & $\leq (n-3)/2$ \\ \hline
     Games-played difference index & 1 & $\geq 1$ & 2 & $\geq 1$ \\ \hline
     Rest difference index & 1 if $n=4$; 2 if $n\geq 6$ & $\geq 1$ & $(n+1)/2$ & $\geq 1$ \\
    \hline
    \end{tabular}
    \vspace{5mm}
    \caption{Summary of our results. All bounds are known to be attainable except that for the rest difference index when $n$ is even. See also Section \ref{sec:discussion} for further discussion.}
    \label{table:summary}
\end{table}


\section{Preliminaries}
\label{sec:prelim}

We assume that the tournament in consideration is a single round-robin tournament, i.e., every pair of teams play each other exactly once. As we will mention in Section \ref{sec:discussion}, however, several of our results can be generalized to arbitrary round-robin tournaments as well.

Let $n$ denote the number of teams in the tournament. We divide the games in the tournament into $r$ rounds of $g$ games, where the first round comprises the first $g$ games, the second round the next $g$ games, and so on. The parameters $r$ and $g$ depend on $n$ and are given by
\[g=\left\lfloor{\frac n2}\right\rfloor,\]
\[
r = 2\cdot\left\lceil{\frac n2}\right\rceil-1 = 
\begin{cases} 
n &\mbox{if } n \mbox{ is odd;} \\ 
n-1 & \mbox{if } n \mbox{ is even.} 
\end{cases} 
\]
A team is said to play in slot $i$ in a round if it plays the $i$th game of that round. We emphasize that in asynchronous tournaments, rounds do not carry any particular meaning in the implementation of the tournament and are defined merely for the sake of convenience of our analysis.

A single round-robin tournament consists of $\binom{n}{2}=\frac{n(n-1)}{2}$ games. Each team plays $n-1$ games, and we have the identity
\[r\cdot g = \left(2\cdot\left\lceil{\frac n2}\right\rceil-1\right)\cdot\left\lfloor{\frac n2}\right\rfloor = \frac{n(n-1)}{2}.\]

A well-known method for scheduling a round-robin tournament, described for instance by Haselgrove and Leech \cite{HaselgroveLe77}, is called the \textit{circle design}. The method works as follows. Assume first that $n$ is even. We arrange the teams into two rows of $n/2$ teams in such a way that the two rows align team by team. The games in the first round correspond to the pairs of teams that are aligned in this arrangement. For asynchronous tournaments, we read the games from left to right. To generate the games in the next round, we keep the top-left team fixed and rotate the remaining teams one step counterclockwise. (It is also possible to rotate the remaining teams one step \textit{clockwise}, but this results in the same schedule as rotating counterclockwise under appropriate renaming of the teams.) We perform the rotation $n-2$ times to generate the games in all $n-1$ rounds. If $n$ is odd, we simply pretend that the top-left team is a dummy team, and whichever team is matched to that team ``sits out'' the round (i.e., gets a bye in that round). The first three rounds for the tournaments with $n=10$ and $n=11$ are shown in Figures \ref{fig:circledesign10} and \ref{fig:circledesign11}, respectively.

\begin{figure}
\centering
\begin{tabular}{ccccc}
1 & 2 & 3 & 4 & 5 \\
10 & 9 & 8 & 7 & 6 \\
\end{tabular}
\hspace{7mm}
\begin{tabular}{ccccc}
1 & 10 & 2 & 3 & 4 \\
9 & 8 & 7 & 6 & 5 \\
\end{tabular}
\hspace{7mm}
\begin{tabular}{ccccc}
1 & 9 & 10 & 2 & 3 \\
8 & 7 & 6 & 5 & 4 \\
\end{tabular}
\caption{The first three rounds generated by the circle design for a tournament with $n=10$.}
\label{fig:circledesign10}
\end{figure}

\begin{figure}
\centering
\begin{tabular}{ccccc}
1 & 2 & 3 & 4 & 5 \\
10 & 9 & 8 & 7 & 6 \\
\end{tabular}
\hspace{7mm}
\begin{tabular}{ccccc}
11 & 1 & 2 & 3 & 4 \\
9 & 8 & 7 & 6 & 5 \\
\end{tabular}
\hspace{7mm}
\begin{tabular}{ccccc}
10 & 11 & 1 & 2 & 3 \\
8 & 7 & 6 & 5 & 4 \\
\end{tabular}
\caption{The first three rounds generated by the circle design for a tournament with $n=11$. Note that one team ``sits out'' each round (i.e., gets a bye in that round).}
\label{fig:circledesign11}
\end{figure}

We now define three measures of a schedule for an asynchronous tournament that concern the quality and fairness of the tournament. The first measure, guaranteed rest time, considers the minimum amount of time that the schedule allows teams to take a rest before their next game.

\begin{definition}
The \textit{guaranteed rest time} of a schedule for an asynchronous tournament is the maximum integer $b$ such that in the schedule, any two games involving a team is separated by at least $b$ games not involving that team.
\end{definition}

A schedule with a high guaranteed rest time is desirable, as it allows teams to take a long rest and prepare themselves for the next game. The higher the guaranteed rest time, the more likely we will see teams perform at their full potential in the tournament.

The next two measures, the games-played difference index and the rest difference index, reflect the fairness of the schedule.

\begin{definition}
The \textit{games-played difference index} of a schedule for an asynchronous tournament is the minimum integer $p$ such that at any point in the schedule, the difference between the number of games played by any two teams is at most $p$.
\end{definition}

It is evident that for any tournament with at least three teams, the games-played difference index is at least $1$. A schedule with a low games-played difference index ensures that all teams have played roughly the same number of games at any point during the tournament. This prevents the advantage that some teams may have if they know the results of too many games involving other teams. Indeed, with this knowledge the teams can adjust their strategy to achieve their desired position in the tournament and may even conspire with one another to do so. 

\begin{definition}
The \textit{rest difference index} of a schedule for an asynchronous tournament is the minimum integer $d$ such that for any game in the schedule, if one team has not played in $i_1$ consecutive games since its last game and the other team has not played in $i_2$ consecutive games since its last game, then $|i_1-i_2|\leq d$. (To handle the situation in which a team is playing its first game in the tournament, we will assume that all teams are involved in an imaginary game that takes place one slot before the first game of the schedule.)
\end{definition}

It is again evident that for any tournament with at least three teams, the rest difference index is at least $1$. A schedule with a low rest difference index guarantees that the two teams involved in a game have approximately the same amount of rest time going into the game.


\section{Even number of teams}
\label{sec:evennum}

In this section, we assume that the tournament in consideration consists of an even number of teams. We will show that under this assumption, the schedule generated by the circle design fares extremely well with respect to all of the measures introduced in Section \ref{sec:prelim}. Since the round-robin tournament with two teams consists of a single game, we will only consider $n\geq 4$. 

We begin by showing that an upper bound on the guaranteed rest time.

\begin{proposition}
\label{prop:evenresttime}
Let $n\geq 4$. Any schedule for a tournament with $n=2k$ teams has guaranteed rest time at most $k-2$.
\end{proposition}

\begin{proof}
Assume for the sake of contradiction that the guaranteed rest time is at least $k-1$. This means that all teams play at most once in any $k$ consecutive games. Since there are $2k$ teams, each team plays exactly once in the first round. Hence the first game in the second round must involve the same two teams as the first game in the first round, a contradiction.
\end{proof}

Next, we analyze the schedule generated by the circle design.

\begin{proposition}
\label{prop:evencircle}
Let $n\geq 4$. The schedule generated by the circle design for a tournament with $n=2k$ teams has guaranteed rest time $k-2$, games-played difference index $1$, and rest difference index $2$ if $n\geq 6$ and $1$ if $n=4$.
\end{proposition}

\begin{proof}
We verify each of the measures separately.

\begin{itemize}
\item \textit{Guaranteed rest time}: Note that each team plays exactly once in every round. Since the slot of a team is shifted by at most $1$ from one round to the next, and each round consists of $k$ slots, the team has a rest of at least $k-2$ games. On the other hand, a team whose slot is shifted to the left has a rest of exactly $k-2$ games.

\item \textit{Games-played difference index}: Since each team plays exactly once in every round, all of the teams have played the same number of games at the end of each round. It follows that the index is $1$.

\item \textit{Rest difference index}: One can directly verify that the index is $1$ if $n=4$. Assume now that $n\geq 6$, and consider the second slot in the second round. One of the teams in that slot is shifted from the third slot in the first round and the other team from the first slot in the first round. Hence the index is at least $2$. On the other hand, the slot of a team is shifted by at most $1$ from one round to the next, so the index is exactly $2$.
\end{itemize}

Thus we have verified all of the measures.
\end{proof}

Propositions \ref{prop:evenresttime} and \ref{prop:evencircle} together imply that the schedule generated by the circle design has an optimal guaranteed rest time and an optimal games-played difference index. Moreover, the rest difference index can be improved by at most $1$. We now show that unless $n=4$, it is impossible to simultaneously obtain a guaranteed rest time of $k-2$ and games-played difference and rest difference indices of $1$.

\begin{theorem}
\label{thm:evennooptimal}
Let $n\geq 6$. No schedule for a tournament with $n=2k$ teams has guaranteed rest time $k-2$, games-played difference index $1$, and rest difference index $1$.
\end{theorem}

\begin{proof}
Assume for the sake of contradiction that such a schedule exists. We first claim that in the schedule, each team plays exactly once in every round. This can be shown by induction on the number of rounds. Suppose that each team plays exactly once in every round up to round $i\geq 0$. In round $i+1$, if some team plays twice, then some other team does not play at all, contradicting the fact that the games-played difference index is $1$. Hence each team also plays exactly once in round $i+1$, completing the induction.

Suppose now that in the first round, teams 1 and 2 play in the first game, teams 3 and 4 in the second, and teams 5 and 6 in the third. Since the guaranteed rest time is $k-2$, the first game in the second round can only involve teams from the first two games in the first round. Assume without loss of generality that team 1 plays team 3 in that game. Similarly, the second game in the second round can only involve teams from the first three games in the first round. The game cannot be played between team 4 and one of teams 5 and 6, since the game involving team 2 in the second round would violate the rest difference index condition. Hence the game is played between teams 2 and 4.

By the same reasoning, the first game in the third round must be played by teams 1 and 4, and the second by teams 2 and 3. But then no team can play against team 1 in the fourth round without violating the rest difference index condition. Thus we have the desired contradiction.
\end{proof}

Theorem \ref{thm:evennooptimal} implies that if a schedule were to have rest difference index $1$, it would have to sacrifice either the guaranteed rest time or the games-played difference index. Nevertheless, it is interesting to ask whether there exists for all even $n$ a schedule with rest difference index $1$. For $n=6$, two such schedules are shown in Figure \ref{fig:restdifference6}. The first schedule also has an optimal guaranteed rest time of $1$, but makes the necessary sacrifice by having a games-played difference index of $2$. On the other hand, the second schedule is worse off in both measures, having guaranteed rest time $0$ and games-played difference index $3$.

\begin{figure}
\centering
\begin{tabular}{*{19}{c}}
1 & 3 & 1 && 2 & 1 & 4 && 1 & 2 & 5 && 1 & 2 & 3 && 2 & 3 & 4 \\
2 & 4 & 5 && 6 & 3 & 5 && 6 & 3 & 6 && 4 & 5 & 6 && 4 & 5 & 6 \\ [3mm]
1 & 3 & 5 && 1 & 1 & 3 && 1 & 2 & 1 && 2 & 3 & 2 && 2 & 4 & 4 \\
2 & 4 & 6 && 3 & 5 & 6 && 6 & 4 & 4 && 6 & 5 & 3 && 5 & 6 & 5 \\
\end{tabular}
\caption{Two schedules with rest difference index $1$ for a tournament with $n=6$.}
\label{fig:restdifference6}
\end{figure}


\section{Odd number of teams}
\label{sec:oddnum}

In this section, we assume that the tournament in consideration consists of an odd number of teams. We will show that unlike in the case with an even number of teams, the schedule generated by the circle design does not fare so well with respect to the measures introduced in Section \ref{sec:prelim}. Nevertheless, we will exhibit a different schedule that performs optimally with respect to all of the measures.

The round-robin tournament with three teams consists of three games, and any two schedules of the three games are equivalent under renaming of the teams, so we have no choice to make in this case. 

We begin by showing an upper bound of $k-1$ for the guaranteed rest time of any schedule. 

\begin{proposition}
\label{prop:oddbound}
Let $n\geq 3$. Any schedule for a tournament with $n=2k+1$ teams has guaranteed rest time at most $k-1$.
\end{proposition}

\begin{proof}
Consider the first $k+1$ games of the tournament. Since they involve the participation of $2k+2$ teams (counting multiplicity), the pigeonhole principle implies that some team plays at least twice among those games. Such a team has a rest of at most $k-1$ games.
\end{proof}

Next, we analyze the schedule generated by the circle design. Even though this schedule does not match the bound in Proposition \ref{prop:oddbound}, we will later exhibit a different schedule that does attain the upper bound.

\begin{proposition}
Let $n\geq 5$. The schedule generated by the circle design for a tournament with $n=2k+1$ teams has guaranteed rest time $k-2$, games-played difference index $2$, and rest difference index $k+1$.
\end{proposition}

\begin{proof}
We verify each of the measures separately.

\begin{itemize}
\item \textit{Guaranteed rest time}: Note that each team plays at most once in every round. Since the slot of a team is shifted by at most $1$ from one round to the next, and each round consists of $k$ slots, the team has a rest of at least $k-2$ games. On the other hand, a team whose slot is shifted to the left has a rest of exactly $k-2$ games.

\item \textit{Games-played difference index}: Since the team that sits out each round is distinct, the difference between the highest and lowest number of games played by a team at the end of each round is $1$. Each team plays at most once in every round, so the difference increases by at most $1$ during a round. Hence the index is at most $2$. 

On the other hand, consider the point after the first game in the third round has just finished. The team that sat out the second round has played once, while a team involved in the first game of the third round has played three times. It follows that the index is $2$.

\item \textit{Rest difference index}: Consider the first slot in the third round. One of the teams in that slot was last involved in the first game of the first round, while the other team played in the second slot of the second round. Hence the index is at least $k+1$. 

On the other hand, consider any two teams involved in a game. If the two teams also played in the previous round, the difference in their rest time is at most $k-1$. Otherwise, one of the team sat out in the previous round. This implies that the team played the first game of the round before the previous round, while the other team played the second game of the previous round. Hence the index is exactly $k+1$.
\end{itemize}

Thus we have verified all of the measures.
\end{proof}

We now show that if a schedule attains the upper bound on the guaranteed rest time, it will also fare optimally with respect to the rest difference index.

\begin{lemma}
\label{lem:oddproperty}
Let $n\geq 3$. Any schedule for a tournament with $n=2k+1$ teams with guaranteed rest time $k-1$ has rest difference index $1$.
\end{lemma}

\begin{proof}
Suppose that a schedule for a tournament with $2k+1$ teams has guaranteed rest time $k-1$. This means that any $k$ consecutive games in the schedule are played by $2k$ distinct teams. 

We show that the rest difference index is $1$. Consider an arbitrary game after the $k$th game. This game cannot involve a team that played in one of the previous $k-1$ games. Moreover, the game cannot be played between the two teams that played each other $k$ games ago. Hence the only possibility is that the game is played between the team that sat out the previous $k$ games and one of the two teams that played $k$ games ago. In particular, all $2k+1$ teams appear in any block of $k+1$ consecutive games. This implies that the team that sat out the previous $k$ games played $k+1$ games ago (if this game exists). Hence the rest difference index is $1$, as desired.
\end{proof}

\begin{figure}
\centering
\begin{tabular}{*{14}{c}}
1 & 3 && 1 & 2 && 4 & 1 && 2 & 3 && 1 & 2 \\
2 & 4 && 5 & 3 && 5 & 3 && 4 & 5 && 4 & 5 \\ 
\end{tabular}
\caption{The schedule as described in Theorem \ref{thm:oddconstruction} with guaranteed rest time $1$, games-played difference index $1$, and rest difference index $1$ for a tournament with $n=5$.}
\label{fig:example5}
\end{figure}

\begin{figure}
\centering
\begin{tabular}{*{28}{c}}
1 & 3 & 5 && 1 & 2 & 4 && 6 & 1 & 2 && 4 & 3 & 1 && 2 & 4 & 3 && 1 & 2 & 5 && 3 & 1 & 2 \\
2 & 4 & 6 && 7 & 3 & 5 && 7 & 3 & 5 && 6 & 7 & 5 && 6 & 7 & 5 && 6 & 4 & 7 && 6 & 4 & 7 \\ [3mm]
1 & 3 & 5 && 1 & 2 & 4 && 6 & 1 & 2 && 4 & 1 & 3 && 2 & 4 & 1 && 3 & 2 & 1 && 5 & 3 & 2 \\
2 & 4 & 6 && 7 & 3 & 5 && 7 & 3 & 5 && 7 & 6 & 5 && 7 & 6 & 5 && 7 & 6 & 4 && 7 & 6 & 4 \\
\end{tabular}
\caption{Two schedules with guaranteed rest time $2$, games-played difference index $1$, and rest difference index $1$ for a tournament with $n=7$. The first schedule corresponds to the one described in Theorem \ref{thm:oddconstruction}.}
\label{fig:example7}
\end{figure}

Proposition \ref{prop:oddbound} and Lemma \ref{lem:oddproperty} do not carry much meaning on their own. Indeed, without an example to show that the bounds can be achieved, it is difficult to tell how useful the bounds are. In particular, the rest difference index of the schedule generated by the circle design, $k+1$, is quite far off from the bound we have so far, $1$. All of these observations raise the natural question of whether there exist other schedules that perform better on some or all measures. The next theorem gives the most satisfying answer possible to this question. It shows that there exists a schedule that fare optimally---and strictly better than the circle-design schedule---with respect to all three measures.

\begin{theorem}
\label{thm:oddconstruction}
Let $n\geq 3$. There exists a schedule for a tournament with $n=2k+1$ teams with guaranteed rest time $k-1$, games-played difference index $1$, and rest difference index $1$.
\end{theorem}

\begin{proof}
In light of Lemma \ref{lem:oddproperty}, it suffices to show the existence of a schedule for a tournament with $2k+1$ teams with guaranteed rest time $k-1$ and games-played difference index $1$. We exhibit the schedule by specifying the slot that the teams play in each round. Slots are taken modulo $k+1$, and slot $0$ means that a team sits out that round. The schedule is defined as follows.

\begin{itemize}
\item For $1\leq i\leq k$, team $2i-1$ is placed in slot $i$ in the first $2i$ rounds. After that, the team moves forward by one slot in each round.
\item For $1\leq i\leq k$, team $2i$ is placed in slot $i$ in the first round. The team moves forward by one slot in each round until round $2k+3-2i$. After that, it stays in the same slot until the last round.
\item Team $2k+1$ is placed in slot $\lfloor j/2\rfloor$ in the $j$th round.
\end{itemize}

The resulting schedules for the tournaments with $n=5$ and $n=7$ can be seen in Figures \ref{fig:example5} and \ref{fig:example7}, respectively.

We show that the schedule is well-defined by demonstrating that every pair of teams play each other exactly once. We divide the verification into cases.

\begin{itemize}
\item For $1\leq i\neq j\leq k$, teams $2i-1$ and $2j-1$ play each other in round $i+j$.
\item For $1\leq i\neq j\leq k$, teams $2i$ and $2j$ play each other in round $2k+3-i-j$.
\item For $1\leq i\leq j\leq k$, teams $2i-1$ and $2j$ play each other in round $1$ if $i=j$ and round $2k+2+i-j$ otherwise.
\item For $1\leq i<j\leq k$, teams $2i$ and $2j-1$ play each other in round $j-i+1$.
\item For $1\leq i\leq k$, teams $2i-1$ and $2k+1$ play each other in round $2i$.
\item For $1\leq i\leq k$, teams $2i$ and $2k+1$ play each other in round $2k+2-i$.
\end{itemize}

Next, we show that the guaranteed rest time is $k-1$. If a team sits out a round between two of its games, then the two games are separated by at least $k$ other games. Otherwise, a team either stays in the same slot or moves one slot forward in the next round. In both cases, the team has a rest of at least $k-1$ games in between.

Finally, we show that the games-played difference index is $1$. At the end of each round, the difference between the highest and lowest number of games played by a team is at most $1$. The teams with a lower number of games played are exactly those that already sat out a round. Hence it suffices to show that in any round, a team that already sat out a round appears no later than a team that participated in every round. One can check that team $2k+1$, which sat out the first round, appears no later than any team that did not sit out, and any other team that already sat out appears no later than it. This completes the proof of the claim, and therefore the theorem.
\end{proof}

The schedule described in Theorem \ref{thm:oddconstruction} is not necessarily the unique schedule satisfying the desired properties. Indeed, for $n=6$, another schedule satisfying the desired properties is shown in Figure \ref{fig:example7}. To see that the two schedules cannot be obtained from each other by permuting the team indices, observe that the first two rounds of games uniquely determine the identity of the teams: team 1 plays in games 1 and 4, team 2 plays in games 1 and 5, team 3 plays in games 2 and 5, and so on. Since the two schedules differ in the second game of the fourth round, no permutation of team indices in one schedule results in the other schedule.

Now that Theorem \ref{thm:oddconstruction} gives us a schedule that fare optimally on all three measures, we may demand a stronger notion of fairness. In particular, while the rest difference index of $1$ guarantees that two teams going into a game have roughly the same amount of rest, it seems fairer if all teams sometimes get a longer rest than their opponent and sometimes a shorter one than if some teams always get a longer rest than their opponent. Nevertheless, the following proposition shows that as long as we insist on maximal guaranteed rest time, this goal cannot be achieved.

\begin{proposition}
\label{prop:alwayswin}
Let $n\geq 3$. For any schedule for a tournament with $n=2k+1$ teams with guaranteed rest time $k-1$, there exists a team that has a longer rest time than its opponent in every game after its first game.
\end{proposition}

\begin{proof}
Consider a schedule for a tournament with $n=2k+1$ teams with guaranteed rest time $k-1$. As in the proof of Lemma \ref{lem:oddproperty}, we find that any game after the $k$th game is played between the unique team that sat out the previous $k$ games and one of the two teams that played $k$ games ago. This implies that if a team just played a game and still has more games left in the tournament, then it will have a rest of either $k-1$ or $k$ games before its next game. Put differently using the terminology in the proof of Theorem \ref{thm:oddconstruction}, a team either stays in the same slot or moves one slot forward in the next round. Since the number of rounds, $2k+1$, is equal to the number of teams, each team sits out exactly one round.

Suppose that teams 1 and 2 play each other in the first game of the tournament, and team 2 has a rest of $k$ games before its next game. We claim that team 2 moves one slot forward in every round. This claim suffices to prove the theorem, since it implies that team 2 has a longer rest time than its opponent in every game after its first game.

Assume for the sake of contradiction that team 2 stays in the same slot at some point during the tournament. Consider the first instance in which this occurs. Since every team sits out exactly one slot, the slot is not slot $0$. 

Suppose that team 2 repeats a slot in rounds $i$ and $i+1$. This means that the team that plays against team 2 in round $i+1$ (say, team $t$) played in the slot before team 2 in round $i$. Since the two teams play each other only once during the tournament, team $t$ also played in the slot before team 2 in round $i-1$, round $i-2$, and so on down to round $2$. Hence team $t$ sat out the first round, played against team 1 in the first slot of the second round, and is in the slot ahead of team 1 in the third round. This implies that team 1 cannot ``overtake'' team $t$ in the slot position for the rest of the tournament. But since team $t$ already sat out while team 1 did not, this means team 1 cannot sit out for the rest of the tournament, a contradiction.
\end{proof}


\section{Discussion}
\label{sec:discussion}

In this paper, we defined three measures that capture quality and fairness properties of a schedule for an asynchronous round-robin tournament, and we also exhibited schedules that perform (close to) optimally with respect to all of these measures. Here we give some comments and directions for future work.

Several of our results can be generalized to arbitrary round-robin tournaments in which every pair of teams play each other a fixed number of times. Indeed, we can turn a single round-robin tournament into an arbitrary round-robin tournament by duplicating each round a desired number of times. This method preserves the guaranteed rest time and rest difference index, and for Proposition \ref{prop:evencircle} it also preserves the games-played difference index. Moreover, Propositions \ref{prop:evenresttime} and \ref{prop:oddbound} can be generalized to arbitrary round-robin tournaments as well.

As mentioned in Section \ref{sec:evennum}, an interesting open question is whether there exists a schedule with rest difference index $1$ when there are an even number of teams. Such schedules are shown in Figure \ref{fig:restdifference6} for the case $n=6$. If the answer turns out to be affirmative, one could also ask for a schedule with a ``balanced'' rest difference in the sense described before Proposition \ref{prop:alwayswin}, i.e., teams sometimes get a longer rest than their opponent and sometimes a shorter one. In addition, one could ask for the optimal value of one measure when the remaining two are forced to achieve their optimal values. From Proposition \ref{prop:evencircle} and Theorem \ref{thm:evennooptimal}, we know that when the guaranteed rest time and the games-played difference index achieve their optimal values, the minimum rest difference index is 2. When the number of teams is odd, it would be interesting to explore whether it is possible to achieve a rest difference index of $1$ with a balanced rest difference if we are willing to sacrifice on other measures.

Finally, it might be worth investigating the structure of ``optimal'' schedules: how many there are, and whether they differ between themselves in some other meaningful way for the teams. This could potentially yield new insights into the fascinating study of scheduling round-robin tournaments.

\subsection*{Acknowledgments.} 

The author thanks the anonymous reviewers for their helpful feedback and acknowledges support from a Stanford Graduate Fellowship.


\end{document}